\documentclass[a4paper,11pt]{article}

\usepackage{amsmath}
\usepackage{amssymb}
\usepackage{amsthm}
\usepackage{graphicx}
\usepackage{amscd}
\usepackage{epic, eepic}
\usepackage{url}
\usepackage{color}
\usepackage[utf8]{inputenc} 
\usepackage{comment}
\usepackage{enumerate}   
\usepackage{graphicx}
\usepackage{epstopdf}
\usepackage{enumitem}
\usepackage{tikz}
\usepackage[top=24mm, bottom=25mm, left=25mm, right=25mm]{geometry}
\usepackage[bookmarks=false,hidelinks]{hyperref}
\usepackage{pgfplots}
\pgfplotsset{compat=1.15}
\usepackage{mathrsfs}
\usetikzlibrary{arrows}
\definecolor{wrwrwr}{rgb}{0.3803921568627451,0.3803921568627451,0.3803921568627451}
\newtheorem{theorem}{Theorem}
\newtheorem{claim}{Claim}

\newtheorem{remark}{Remark}

\newtheorem{conjecture}{Conjecture}

\newtheorem{observation*}{Observation}
\newtheorem{lemma}{Lemma}

\newtheorem{question}{Question}

\newcommand{\ex}{{\rm  ex}}

\usepackage{authblk}

\title{Edges not covered by monochromatic bipartite graphs}
\author[1,2]{Xiutao Zhu}
\author[1]{Ervin Győri}
\author[1,3]{Zhen He}
 \author[1,3]{Zequn Lv \footnote{Correspondoing author. Email: {lvzq19@mails.tsinghua.edu.cn}}}
\author[1,5]{Nika Salia}
\author[1]{Casey Tompkins}
\author[1,4]{Kitti Varga}
\date{}

\affil[1]{Alfr\'ed R\'enyi Institute of Mathematics, Hungarian Academy of Sciences. }
\affil[2]{Department of Mathematics, Nanjing University.}
\affil[3]{Department of Mathematical Sciences, Tsinghua University.}
\affil[4]{Department of Computer Science and Information Theory, Budapest University of Technology and Economics.}
\affil[5]{Extremal Combinatorics and Probability Group, Institute for Basic Science, Daejeon, South Korea.}
\begin{document}
\maketitle
\begin{abstract}
Let $f_k(n,H)$ denote the maximum number of edges not contained in any monochromatic copy of~$H$ in a $k$-coloring of the edges of $K_n$, and let $\ex(n,H)$ denote the Tur\'an number of $H$. 
In place of $f_2(n,H)$ we simply write $f(n,H)$.
In~\cite{Keevash}, Keevash and Sudakov proved that $f(n,H)=\ex(n,H)$ if $H$ is an edge-critical graph or $C_4$ and asked if this equality holds for any graph $H$.
All known exact values of this question require $H$ to contain at least one cycle.
In this paper we focus on acyclic graphs and have the following results:

(1) We prove $f(n,H)=\ex(n,H)$ when $H$ is a spider or a double broom.

(2) A \emph{tail} in $H$ is a path $P_3=v_0v_1v_2$ such that $v_2$ is only adjacent to $v_1$ and $v_1$ is only adjacent to $v_0,v_2$ in $H$. 
We obtain a tight upper bound for $f(n,H)$ when $H$ is a bipartite graph with a tail. This result provides the first bipartite graphs which answer the question of Keevash and Sudakov in the negative.

(3) Liu, Pikhurko and Sharifzadeh~\cite{Liu} asked if $f_k(n,T)=(k-1)\ex(n,T)$ when $T$ is a tree. 
We provide an upper bound for $f_{2k}(n,P_{2k})$ and show it is tight when $2k-1$ is prime. 
This provides a negative answer to their question.
\end{abstract}

\section{Introduction}
Given any graph $H$, the classical theorem of Ramsey asserts that there exists an integer $R(H,H)$ such that every $2$-coloring of the edges of the complete graph $K_n$ with $n\ge R(H,H)$ contains a monochromatic copy of $H$. 
A natural extension of this problem is determining how many monochromatic copies of $H$ there are.
For the case of $H=K_3$, this question was answered by Goodman~\cite{Goodman} and the case of $H=K_4$ was settled by Thomason~\cite{Thomason}.

In a different direction, one can ask how many edges must be contained in some monochromatic copy of $H$ in every $2$-coloring of the edges of $K_n$ (equivalently how many edges there can be in a $2$-coloring which are not contained in any monochromatic copy of $H$).
The first result about this topic is due to Erd\H{o}s, Rousseau and Schelp~\cite{Erdos}. 
They considered the maximum number of edges not contained in any monochromatic triangle in a $2$-coloring of the edges of~$K_n$.
Erd\H{o}s also wrote ``many further related questions can be asked'' in \cite{Erdos}.
In this paper, we will consider problems of this type.

Let $c$ be a $2$-coloring of the edges of $K_n$ and let $H$ be a graph. If an edge of $K_n$ is not contained in any monochromatic copy of $H$, then we say it is NIM-$H$. 
Let $E(c,H)$ denote the set of all NIM-$H$ edges in $K_n$ under the $2$-edge-coloring $c$ and let
\[f(n,H)=\max\big\{|E(c,H)|:~c \mbox{ is a $2$-edge-coloring of } K_n\big\}.\]
Let $\ex(n,H)$ be the Tur\'an number of $H$.
If one considers a $2$-coloring of the edges of $K_n$ in which one of the colors yields an extremal graph for $H$, then it is easy to see
\begin{equation}\label{eq1.1}
f(n,H)\ge \ex(n,H).
\end{equation}
As observed by Alon, the result on $f(n,K_3)$ by Erd\H{o}s, Rousseau and Schelp~\cite{Erdos} can also be deduced from a result of Pyber~\cite{Pyber} (see~\cite{Keevash}). 
In~\cite{Keevash}, Keevash and Sudakov studied $f(n,H)$ systematically. 
They proved that if $H$ contains an edge $e$ such that $\chi(H-e)<\chi(H)$ or $H=C_4$, then equality holds in~\eqref{eq1.1} for sufficiently large $n$.
Furthermore, they asked if the equality holds for all $H$.
\begin{question}[Keevash, Sudakov~\cite{Keevash}]\label{problem1}
Is it true that for any graph $H$ we have $f(n, H)=\ex(n, H)$ when $n$ is sufficiently large?
\end{question}

In 2017, Ma~\cite{Ma} provided an affirmative answer to Question~\ref{problem1} for an infinite family of bipartite graphs $H$, including all even cycles and complete bipartite graphs $K_{s,t}$ for $t>s^2-3s+3$ or $(s,t)\in \{(3,3), (4,7)\}$.
In 2019, Liu, Pikhurko and  Sharifzadeh~\cite{Liu} extended Ma's result by providing a larger family of bipartite graphs for which $f(n,H)=\ex(n,H)$ holds (however, the graphs they construct still contain a cycle).
Surprisingly, Yuan~\cite{Yuan} recently found an example showing that the assertion in Question~\ref{problem1} does not hold in general.  

\begin{theorem}[Yuan~\cite{Yuan}]\label{Yuan}
Let $p\ge t+1\ge 4$ and $K_t^{p+1}$ denote the graph obtained from $K_t$ by replacing each edge of $K_t$ with a clique $K_{p+1}$. When $n$ is sufficiently large, then
\[f(n,K_t^{p+1})=\ex \big( n, K_t^{p+1} \big) +\binom{\binom{t-1}{2}}{2}.\]
\end{theorem}
Based on this result, he conjectured the following.
\begin{conjecture}[Yuan~\cite{Yuan}] \label{yuanconj}
Let $H$ be any graph and $n$ be sufficiently large. Then there exists a constant $C=C(H)$ such that $f(n,H)=\ex(n,H)+C$.
\end{conjecture}

As mentioned earlier, the known results about the exact value of $f(n,H)$ require that $H$ contains a cycle.  For acyclic graphs and some other bipartite graphs, the situation is less clear.
Thus, in this paper, we will focus on this case. 
A \emph{spider} is the graph consisting of $t$ paths with one common end vertex such that all other vertices are distinct.  
A \emph{double broom} with parameters $t$, $s_1$ and $s_2$ is the graph consisting of a path with $t$ vertices with $s_1$ and $s_2$ distinct leaves appended to each of its respective end vertices.  

\begin{theorem}\label{Thm1}
Let $H$ a spider or a double broom with $s_1<s_2$ and $n$ be sufficiently large, we have
\[f(n,H)=\ex(n, H).\]
\end{theorem}

A \emph{tail} in a (not necessary acyclic) graph $H$  is a path $P_3=v_0v_1v_2$ such that $v_2$ is only adjacent to $v_1$ and $v_1$ is only adjacent to $v_0$ and $v_2$.
\begin{theorem}\label{Thm2}
Let $H=(A,B,E)$ be a bipartite graph containing a tail and $|A|\le |B|$. When $n$ is sufficiently large, we have 
\begin{equation}\label{eq1}
f(n,H)\le \ex(n,H)+\binom{|A|-1}{2}.
\end{equation}
Furthermore, the upper bound is tight. 
\end{theorem}

\begin{remark}
In Theorem~\ref{Thm2}, there are many bipartite graphs $H$ such that $f(n,H)$ achieves an upper bound greater than $\ex(n,H)$. 
This implies that even for the bipartite case, the answer to Question~\ref{problem1} can be negative.
However, the graphs from Theorem~\ref{Thm2} satisfy Conjecture~\ref{yuanconj}.
\end{remark}

\vskip 3mm

We will also consider the case of edge colorings with $3$ or more colors.
Let $f_k(n,H)$ be the maximum number of edges not contained in any monochromatic copy of $H$ in a $k$-coloring of the edges of $K_n$. 
Thus, $f_2(n,H)=f(n,H)$. 
It appears likely that for $k \ge 3$, the function $f_k(n,H)$ has different behavior for bipartite graphs and non-bipartite graphs.
For non-bipartite graphs, one can see that $f_k(n,H)\not= (k-1)\ex(n,H)$ since $(k-1)\ex(n,H)\ge \binom{n}{2}$.

For a tree $T$, Ma~\cite{Ma} constructed a lower bound by taking random overlays of $k-1$ copies of extremal $T$-free graphs, and the construction implies $f_k(n,T)\ge (k-1-o(1))\ex(n,T)$. 
Liu, Pikhurko and Sharifzadeh~\cite{Liu} showed that this lower
bound is asymptotically correct.
\begin{theorem}[Liu, Pikhurko,
Sharifzadeh~\cite{Liu}]
Let $T$ be a tree with $h$ vertices. Then there exists a constant $C(k,h)$ such that for all sufficiently large $n$, we have
\[ \big| f_{k}(n,T)-(k-1)\ex(n,T) \big|\le C(k,h).\]
\end{theorem}

For more general bipartite graph $H$, Ma~\cite{Ma} wrote ``it may be reasonable to ask if $f_k(n,H)=(k-1)\ex(n,H)$ holds for sufficiently large $n$".  However, this is not true for disconnected bipartite graphs. 
Liu, Pikhurko and Sharifzadeh~\cite{Liu} gave an example and showed $f_k(n,2K_2)={(k-1)}\ex(n,2K_2)-\binom{k-1}{2}$. Based on this example, Liu, Pikhurko and Sharifzadeh~\cite{Liu} asked the following question.
\begin{question}[Liu, Pikhurko, Sharifzadeh~\cite{Liu}] \label{problem2}
Is it true that $f_k(n, T)=(k-1)\ex(n, T)$ for any tree $T$ and sufficiently large $n$?
\end{question}

Our third result concerns the case when $T$ is a path with an even number of vertices and yields a negative answer to Question~\ref{problem2}.
\begin{theorem}\label{Thm3}
Let $k \ge 1$ and $n\ge (2k)^{2k^2}$ be integers. We have
\[f_{2k}(n,P_{2k})\le (2k-1)\ex(n, P_{2k})+(k-1)\binom{2k-1}{2}.\]
Furthermore, equality holds when $2k-1$ is a prime and $n\in \big\{ a(2k-1)+{(k-1)}, \; a{(2k-1)}+k \big\}$.
\end{theorem}

\noindent \textbf{Notation and organization.} For a given graph $G$, we use $e(G)$ to denote the number of edges of $G$. For a subset of vertices $X$, let $G[X]$ denote the subgraph induced by $X$ and $G-X$ denote the subgraph induced by $V(G)\setminus X$. For two disjoint subset $X,Y$, let $G[X,Y]$ denote the bipartite subgraph of $G$ consisting of the edges of $G$ with one end vertex in $X$ and the other in~$Y$. In a red-blue edge-colored complete graph $K_n$, we say that $u$ is a red (or blue) neighbor of $v$ if the edge $uv$ is red (or blue). For a set $X$ of vertices, let $N_r(v,X)$ and $N_b(v,X)$ denote the red and blue neighbors of $v$ in $X$, respectively.  Let $d_r(v,X)= \big| N_r(v,X) \big|$ and $d_b(v,X)= \big| N_b(v,X) \big|$.  If $X=V(K_n)$, then we simply  write  $d_r(v)$ and $d_b(v)$. For two graphs $G$ and $H$, we use $G\cup H$ to denote the disjoint union of $G$ and $H$. Let $G+H$ be the graph obtained from $G\cup H$ by adding all edges with one end vertex in $V(G)$ and one end vertex in $V(H)$.

The rest of the paper is organized as follows. In Sections~\ref{sec2} and~\ref{sec3},  we study the function $f(n,H)$ and prove Theorems~\ref{Thm1} and~\ref{Thm2}, respectively. In Section~\ref{sec4}, we study the general function $f_k(n,H)$ and prove Theorem~\ref{Thm3}.

\section{Proof of Theorem \ref{Thm1}} \label{sec2}

Let $H$ be a spider or a double broom on~$k$ vertices and $c$ be a red-blue edge-coloring of $K_n$ with $|E(c,H)|$ being maximum.
If $E(c,H)$ contains no $H$, then  
\[f(n,H)= \big| E(c,H) \big| \le \ex(n,H),\]
and we are done.
Hence we may assume there is a non-monochromatic copy of $H$ in $E(c,H)$.

Since we can take $n$ to be larger than the Ramsey number $R(k^2, k^2)$, it follows, without loss of generality, that $K_n$ contains a blue clique $K$ of size at least $k^2$.  
We partition $V(K_n)$ into two parts $X$ and $Y$ such that $Y$ is maximal with the property that any vertex $v$ in~$Y$ has $d_b(v,Y)\ge k$ and $X$ consists of the remaining vertices.
Note that the large blue clique $K$ is contained in $Y$, and hence $|Y|\ge k^2$.
Since each vertex in $Y$ has blue degree at least $k$ in $Y$, every blue edge in~$Y$ or between $X$ and $Y$ can be extended to  a blue copy of $H$. 
Hence, all blue NIM-$H$ edges are contained in $X$ and $|X|\ge 2$.

For each vertex $u$ in $X$, we have $d_b(u,Y)\le (k-1)$. Thus for each subset $X'$ of $X$, the subset $Y'=Y\setminus N_b(X',Y)$ is such that $K_n[X',Y']$ is a red complete bipartite graph and $|Y'|\ge |Y|-(k-1)|X'|$. We call $Y'$ the corresponding subset of $X'$.

First assume $|X|\ge \left\lfloor{\frac{k}{2}} \right\rfloor+1$.  For each red edge $uv$ contained in $X$ or between $X$ and $Y$, we can find a subset $X'\subseteq X$ of size $\left\lfloor{\frac{k}{2}} \right\rfloor$ that contains exactly one of $u$ and $v$. Using the corresponding subset $Y'$ of $X'$, this red edge $uv$ can be extended to a red copy of $H$. Hence all red NIM-$H$ edges are contained in $Y$ and
\[ \big| E(c,H) \big| \le \ex \big( |Y|,H \big) +\ex \big( |X|,H \big) \le \ex(n,H).\]

Therefore, in the rest of the proof, we will assume $|X|\le \left\lfloor{\frac{k}{2}} \right\rfloor$. Furthermore, each red edge in~$Y$ is NIM-$H$, otherwise we replace the color of this edge by blue and since $E(c,H)$ is maximum, it has no changes.

Next we distinguish two cases based on whether $H$ is a spider or a double broom.

\vskip 3mm

\noindent \textbf{The proof when $H$ is a spider.} Let $H$ be a spider consisting of $t$ paths with a common initial vertex~$v_0$. 
We call each path starting from $v_0$ a branch, and we assume that the lengths of these~$t$ branches are $\ell_1,\ldots,\ell_t$ such that $v(H)=k=1+\sum_{i=1}^{t}\ell_i$.

Now we choose a copy of $H$ from $E(c,H)$ and denote it by $H'$.
Let $X'=X\cap V(H')$. 
Since~$H'$ contains blue edges and all NIM-$H$ blue edges are contained in $X$, we have $X'\not=\emptyset$ and the corresponding subset~$Y'$ is of size at least 
\[|Y|-(k-1)|X'|\ge k.\]

For every branch of $H'$, we apply the following method to replace all blue edges with red edges. 
First, every branch consisting entirely of blue edges is replaced by a red path of the same length in $K_n[X',Y']$.
This can be done since $K_n[X',Y']$ is a complete bipartite graph consisting of only red edges and $Y'$ is large enough.
For any remaining branch $v_0v_1\ldots v_{\ell_m}$, let $v_iv_{i+1}$ be the first red edge on this branch, i.e., every edge in the path $v_0v_1\ldots v_i$ is blue.
If $i$ is even, we replace the path $v_{2j}v_{2j+1}v_{2j+2}$ by a new red path $v_{2j}y_jv_{2j+2}$ with a distinct $y_j\in Y'$ for all $0\le j\le \frac{i}{2}-1$.
If $i$ is odd, we replace the path $v_{2j}v_{2j+1}v_{2j+2}$  by a new red path $v_{2j}y_jv_{2j+2}$ with a distinct $y_j\in Y'$ for all $0\le j\le \frac{i-1}{2}-1$ and replace the single edge $v_{i-1}v_i$ by a new red path $v_{i-1}y'v_i$ with a distinct $y'\in Y'$.
For all other blue edges after $v_iv_{i+1}$, we replace them by a new red $P_3$ with the middle vertices in $Y'$. Again, this can be done since $K_n[X',Y']$ is a complete bipartite graph consisting of only red edges and $Y'$ is large enough.

After this, the original branch becomes a longer red path and we take the first segment of length $\ell_m$ as the new branch.
Note that this new branch still contains the original red edge $v_iv_{i+1}$ unless $i$ is odd and $i+1=\ell_m$.
Let $H''$ be the resulting copy of~$H$.

If $H''$ still contains one of the original red edges, then we have a monochromatic copy of~$H$, a contradiction since the original edges are NIM-$H$.
Otherwise every branch of $H'$ is either entirely blue or has even length and is such that only the final edge is red. 
However, then we have $|X|\ge |X'|\ge \left\lfloor{\frac{k}{2}}\right\rfloor+1$, a contradiction of our assumption that $|X| \le \left\lfloor{\frac{k}{2}} \right\rfloor$ (recall that the blue edges are in $X'$).
The proof is complete for spiders.

\vskip 3mm
\noindent \textbf{The proof when $H$ is a double broom.}
Let $H$ be a double broom with parameters $t$, $s_1$ and $s_2$ such that $k=t+s_1+s_2$ and $s_1< s_2$.  

First, assume that $t$ is odd and $|X|\ge \left\lfloor{\frac{t}{2}}\right\rfloor+1$. 
For a red edge $uv$ with $u\in X,~v\in Y$, there is a subset $X'\subseteq X$ of size $\frac{t+1}{2}$ containing $u$. 
Let $Y'$ be the corresponding subset for $X'$. 
Then there is a path $P_t$ in $K_n[X',Y']$ which starts from $u$ and ends at another vertex, say $w$ in $X'$, and avoids $v$.
Since $|Y'|\ge k^2-(k-1)\frac{t+1}{2}$, we can select additional red edges incident to $u$ and $w$, which together with the edge $uv$ represent the set of edges incident to the leaves of $H$.
It follows that $uv$ is not NIM-$H$.
Hence all red NIM-$H$ edges are contained in $X$ and $Y$, and we have
\[ \big| E(c,H) \big| \le \ex \big( n-|X|, H \big) +\binom{|X|}{2}\le \ex(n,H),\]
where the second inequality holds since $|X|\le \frac{k}{2}$.

Now assume that $t$ is even and $|X|\ge \left\lfloor{\frac{t}{2}}\right\rfloor+1$.
Let $Y_1=\{v\in Y:~ d_r(v,X)\ge 1 \}$ and $Y_2=Y \setminus Y_1$. 
Since each vertex in $X$ has at most $k-1$ blue neighbors in $Y$, we have $|Y_2|\le k-1$.  

Now we show that for each vertex $v \in Y_1$, there are at most $s_1+\frac{t}{2}-1$ NIM-$H$ edges incident to $v$. Suppose by way of contradiction that for a vertex $v\in Y_1$, there are at least $s_1+\frac{t}{2}$ red NIM-$H$ edges incident to $v$. 
By the definition of $Y_1$, there is a red edge $vu$ with $u\in X$. 
Let $X'=X$ and let $Y'\subset Y$ be the corresponding subset of $X'$. We extend the red edge $vu$ to a red path $P_t$ in such a way that: (1) one of the end vertex is $v$ and the other end vertex $w$ is in $X'$, (2) every second vertex of the path is in $X'$ and the remaining vertices of the path are in $Y'$, (3) there remain at least $s_1$ red NIM-$H$ edges incident to $v$ which are not vertices of the path. These conditions can be satisfied since $Y'$ is sufficiently large. 
Now at least $s_1$ red NIM-$H$ edges incident to $v$ are not covered by the vertices of the path, which we can view as leaf edges of $H$ incident to $v$. Select another $t$ red (but not necessarily NIM-$H$) edges incident to $w$ and to some vertices which have not been used yet.
Thus we found a red copy of $H$ containing at least one NIM-$H$ edge, a contradiction. 

Therefore, for each vertex $v \in Y_1$, there are at most $s_1+\frac{t}{2}-1$ NIM-$H$ edges incident to $v$.
All other NIM-$H$ edges are contained in $Y_2$ and $X$. Hence,
\begin{align*}
 \big| E(c,H) \big| \le& |Y_1| \left( s_1+\frac{t}{2}-1 \right)+\binom{|Y_2|}{2}+\binom{|X|}{2} \tag{1}\\
\le & \ex \big( |Y_1|, H \big) +\ex \big( |Y_2|,H \big) +\ex \big( |X|,H \big)\\
\le &\ex(n,H),  
\end{align*}
 where the second inequality holds since the coefficient of $|Y_1|$ satisfies $s_1+\frac{t}{2}-1<\frac{k-2}{2}$ and $|Y_2|\le k-1$, $|X|\le \frac{k}{2}$. Thus, we are done in the case $|X|\ge \left\lfloor{\frac{t}{2}} \right\rfloor+1$.

Finally, we consider the case when $|X|\le \left\lfloor{\frac{t}{2}} \right\rfloor$.
Since $|X|\ge 2$, we have $t\ge 4$.
Let  $Y_1=\{ {v\in Y:}~ d_r(v,X)\ge 2\}$  and $Y_2=Y\setminus Y_1$. 
Now we show that there is no red path of length $t-2|X|+1$ in $Y_1$. Suppose by way of contradiction that $P$ is a red path of length $t-2|X|+1$ in~$Y_1$. 
First, we extend $P$ to a red path of length $t-1$ using vertices in $X$ and the corresponding subset of $X$ in $Y$  such that the two end vertices of this longer path, say $u$ and $v$, are contained in $X$. 
Since each vertex in $X$ has red degree at least $|Y|-(k-1)$ in $Y$, we can find  $s_1$ new red neighbors of $u$ and $s_2$ new red neighbors of $v$ in $Y$ and view them as the leaf-edges of $H$. 
That is, we extended the red path $P$ to a red copy of $H$. However, as we assumed all red edges in $Y$ are NIM-$H$, we have a contradiction.

Now we show $|Y_2| \le s_1 -1$. Suppose by way of contradiction that  $|Y_2|\ge s_1$.
If there are two vertices $v_1,v_2$ in $Y_2$ such that $N_b(v_1, X)\cup N_b(v_2,X)=X$, then for any blue edge $u_1u_2$ in~$X$, we have that $v_1u_1u_2v_2$ or $v_1u_2u_1v_2$ is a blue path. 
Since $t\ge 4$ and all vertices in $Y$ have large blue degree in $Y$, this blue path can be extended to a blue copy of $H$.
Hence there are no blue NIM-$H$ edges, a contradiction. 
Thus by the definition of $Y_2$, there exists a vertex $w\in X$ such that $N_b(v,X)=X\setminus\{w\}$ for any $v\in Y_2$. 
Let $uu'$ be a blue NIM-$H$ edge in $X$ with $u\ne w$.
Using $uu'$ and $s_1$ blue edges between $u$ and $Y_2$, we can find a blue star with $s_1+1$ leaves.
By the definition of $Y$,  we can extend this blue star to a blue copy of $H$ using other vertices in $Y$, a contradiction.
Hence we have $|Y_2|\le s_1-1$. Furthermore, there are at most $|Y_2|$ red NIM-$H$ edges between $X$ and $Y_2$.

Therefore, we have
\begin{align*}
\big| E(c,H) \big| &\le \ex \big( |Y_1|, P_{t-2|X|+2} \big)+ |Y_1| \big( |Y_2|+|X| \big)+\binom{|Y_2|}{2}+\binom{|X|}{2}+|Y_2|\\
& \le \frac{t-2|X|}{2}|Y_1|+|Y_1| \big( |Y_2|+|X| \big) +\binom{|Y_2|}{2}+\binom{|X|}{2}+|Y_2|\\
&\le \frac{t+2(s_1-1)}{2} \big( n-(s_1-1)-|X| \big) +\binom{s_1}{2}+\binom{|X|}{2}\\
&\le \frac{t+2s_1-2}{2}n\le \ex(n,H), \tag{2}
\end{align*}
where the last inequality holds since $s_1<s_2$. The proof is complete. $\hfill\blacksquare$

\begin{remark}
One may note that in inequality $(1)$ and $(2)$, we need the condition $s_1<s_2$ to ensure that $\frac{t+2s_1-2}{2}n\le \ex(n,H)$. For the case $s_1=s_2$, these inequalities still show $f(n,H)\le \frac{k-2}{2}n$ but this does not imply $f(n,H)\le \ex(n,H)$ for all $n$.
With additional details, one could extend the proof to the case $s_1=s_2$. 
But this would make our proof more complicated, so we omit it.    
\end{remark}



\section{Proof of Theorem~\ref{Thm2}} \label{sec3}

We first construct some bipartite graphs which attain the upper bound in~\eqref{eq1}. 
Our idea comes from a theorem of Bushaw and Kettle~\cite{Kettle}. Before we present the detailed constructions, we recall some results which we will require.

It is well-known that $\ex(n,T)\le \frac{v(T)-2}{2}n$ when $T$ is a path or star. For a general tree $T$, this is the celebrated  Erd\H{o}s--S\'{o}s Conjecture.
\begin{conjecture}[Erd\H{o}s--S\'{o}s]
For a tree $T$, we have $\ex(n, T)\le \frac{v(T)-2}{2}n$.
\end{conjecture}
In 2005, McLennan~\cite{mclennan2005erdHos} proved that the Erd\H{o}s--S\'os Conjecture holds for trees of diameter at most four.
\begin{theorem}[McLennan \cite{mclennan2005erdHos}]\label{dim 4}
 Let T  be a tree of diameter at most four, then $\ex(n,T)\le \frac{v(T)-2}{2}n$.
\end{theorem}

A tree is called balanced if it has the same number of vertices in each color class when the tree is viewed as a bipartite graph. 
A forest is called balanced if  each of its components is a balanced tree. 
Bushaw and Kettle~\cite{Kettle} proved the following theorem.
\begin{theorem}[Bushaw and Kettle~\cite{Kettle}]\label{Kettle}
Let $H$ be a balanced forest on $2a$ vertices which comprises at least two trees. If the Erd\H{o}s--S\'{o}s Conjecture holds for each component tree in $H$, then for any $n\ge 3a^2+32a^2\binom{2a}{a}$, we have
\[
\ex(n,H)=\begin{cases}
\binom{a-1}{2}+(a-1)(n-a+1) & \text{if $H$ admits a perfect matching,}\\
(a-1)(n-a+1) & \text{otherwise.}
\end{cases}
\]
\end{theorem}

\vskip 3mm
Now, making use of Theorems~\ref{dim 4} and~\ref{Kettle}, we construct some bipartite graphs $H$ which are negative examples for Question~\ref{problem1}. 
Let $\mathcal{H}_1$ be the family  of all balanced trees on $2a$ vertices which admit no perfect matching and for which the Erd\H{o}s--S\'{o}s Conjecture holds.
One can see that $\mathcal{H}_1$ is not empty since a double star $S_{a-1,a-1}$ is a balanced tree on $2a$ vertices and the Erd\H{o}s--S\'{o}s Conjecture holds for it by Theorem~\ref{dim 4}. 
Let $\mathcal{H}_2$ be the family of balanced trees on $2a$ vertices for which the Erd\H{o}s--S\'{o}s Conjecture holds for sufficiently large $n$.
Note that~$\mathcal{H}_2$ is also nonempty, for example a path on $2a$ vertices belongs to $\mathcal{H}_2$.

Let $H_1\in \mathcal{H}_1$, $H_2\in \mathcal{H}_2$ and set $H=H_1\cup H_2$. 
We know that $H$ is a balanced forest on $4a$ vertices. Since $H_1$ admits no perfect matching, $H$ admits no perfect matching either. The Erd\H{o}s--S\'{o}s Conjecture holds for each component of $H$, hence by Theorem \ref{Kettle}, when $n$ is sufficiently large, we have
\[\ex(n,H)=(2a-1)(n-2a+1).\]

On the other hand, consider a partition of the vertices of the complete graph $K_n$ into parts $X$ and $Y$ with $|X|=2a-1$ and $|Y|=n-2a+1$. 
We color all edges between $X$ and $Y$ red and the remaining edges blue. 
One can see that the red edges induce a complete bipartite graph $K_{2a-1,n-2a+1}$ which contains no red copy of $H$.
The blue edges induce a blue $(2a-1)$-clique and a blue $(n-2a+1)$-clique which are disjoint with each other.
Since each component of $H$ contains $2a$ vertices, all blue copies of $H$ are contained in the $(n-2a+1)$-clique.
Therefore, all red edges and all the edges in the blue $(2a-1)$-clique are NIM-$H$, that is,
\[f(n,H)\ge \binom{2a-1}{2}+(2a-1)(n-2a+1)=\binom{2a-1}{2}+\ex(n,H).\]
Therefore, such a bipartite graph $H$ attains the upper bound of the inequality~\eqref{eq1}.

\vskip 3mm
Next we prove that if the bipartite graph $H$ contains a tail $v_0v_1v_2$, then $f(n,H)\le \ex(n,H)+\binom{|A|-1}{2}$. 
Note that it is possible that $H$ is disconnected, hence let $H=H_1\cup \dots \cup H_q$, where $H_i$ are its components (if $H$ is connected, then $H=H_1$) and we say the tail $v_0v_1v_2$ is contained in $H_1$.
Let $A_i,B_i$ be the two color classes of $H_i$ with $|A_i|\le |B_i|$ for any $1 \le i \le q$, and let $A=\bigcup _{i=1}^qA_i$, $B=\bigcup^q _{i=1}B_i$. Set $a=|A|$.

Since we take $n$ to be sufficiently large, we may assume $n\ge R \big( K_{v(H)}, K_{v(H)} \big)$. Let $c$ be a red-blue edge-coloring of $K_n$. Without loss of generality, there is a blue clique on at least $v(H)$ vertices in $K_n$.
Let $K_t$ be a blue clique in $K_n$ such that $t$ is as large as possible.
We have $t\ge v(H)$ and every other vertex has a red neighbor in $V(K_t)$.
We partition  $V(K_n)\setminus V(K_t)$ into two subsets $X,Y$ such that $Y$ consists of the vertices which have blue neighbors in $V(K_t)$ and~$X$ consists of the remaining vertices. Hence all edges between $V(K_t)$ and~$X$ are red.

The following claims will be used several times.
\begin{claim}\label{claim1}
All  blue NIM-$H$ edges  are contained in $X$.
\end{claim}
\begin{proof}

 Obviously, the blue edges in $K_t$ and $K_n \big[ V(K_t),Y \big]$ are not NIM-$H$. 
 Let $xy$ be a blue edge with $y\in Y$ and $x\in X\cup Y$. 
 By the definition of $Y$, the vertex $y$ has a blue neighbor, say~$v$, in $V(K_t)$. If we embed $V(H)\setminus \{v_1,v_2\}$ into $V(K_t)$ and view $vyx$ as the tail of $H$, then we find a blue copy of $H$ containing $xy$. Thus $xy$ is not NIM-$H$.
Therefore, all  blue NIM-$H$ edges  are contained in~$X$.
\end{proof}

\begin{claim}\label{claim2}
 If $|X|\ge a$, then the red edges between $X$ and $V(K_t)\cup Y$ are not NIM-$H$.
\end{claim}
\begin{proof}
Since the red edges between $X$ and $V(K_t)$ induce a red complete bipartite graph and $|X|\ge a$ and $t\ge v(H)$, each such edge is contained in a red copy of $H$, thus these edges are not NIM-$H$.
Let $xy$ be a red edge with $x\in X$, $y\in Y$.
By the maximality of $K_t$, the vertex $y$ has a red neighbor, say $v$, in $V(K_t)$. Actually,~$\{x,y,v\}$ induces a red triangle. 
If the tail $v_0v_1v_2$ of~$H$ satisfies $\{v_0,v_2\}\subset B$ and $v_1\in A$, then embed $B\setminus\{v_2\}$ into $V(K_t)$ so that $v_0$ is identified with $v$, embed $A\setminus \{v_1\}$ into $X\setminus\{x\}$ and view $vxy$ as the tail of $H$, thus we find a red copy of $H$ containing $xy$. So in this case, $xy$ is not NIM-$H$. 
If the tail $v_0v_1v_2$ of $H$ satisfies $\{v_0,v_2\}\subset A$ and $v_1\in B$, then embed $B\setminus\{v_1\}$ into $V(K_t)\setminus \{v\}$, embed $A\setminus \{v_2\}$ into $X$ so that $v_0$ is identified with $x$. View $xyv$ as the tail,  we  find a red copy of $H$ containing $xy$. So in this case, $xy$ is not NIM-$H$ either.
\end{proof} 

We distinguish three cases based on the size of $X$.

{\bf Case 1:} $|X|\ge a+1$. In this case, we first claim that the red edges in $X$ are also not  NIM-$H$.  Let $xx'$ be a red edge contained in $X$ and $v$ be a vertex in $K_t$. If the tail $v_0v_1v_2$ in $H$ satisfies $\{v_0,v_2\}\subset B$ and $v_1\in A$, then since $\big| X \setminus \{x,x'\} \big| \ge a-1= \big| A \setminus \{v_1\} \big|$,
we can embed $A\setminus \{v_1\}$ into $X\setminus\{x,x'\}$, embed $B\setminus \{v_2\}$ into $V(K_t)$ so that $v_0$ is identified with $v$ and view $vxx'$ as the tail $v_0v_1v_2$, thereby finding a red copy of $H$ containing $xx'$. So in this case, $xx'$ is not NIM-$H$.
If the tail $v_0v_1v_2$ in $H$ satisfies $\{v_0,v_2\}\subset A$ and $v_1\in B$, then we embed $A\setminus\{v_2\}$ into $X\setminus\{x'\}$ so that $v_0$ is identified with $x$, embed $B\setminus\{v_1\}$ into $V(K_t)\setminus\{v\}$ and view $xx'v$ as the tail, and again we can find a red copy of $H$ containing $xx'$. Therefore, $xx'$ is not NIM-$H$.

By Claim~\ref{claim2} and the above result, all red NIM-$H$ edges are contained in $V(K_t)\cup Y$. 
Note that the red NIM-$H$ edges contained in $V(K_t)\cup Y$ induce an $H_1$-free graph. Otherwise, such a red copy of $H_1$ together with a red copy of $H_2\cup\dots \cup H_q$ (if $H$ is disconnected) contained in the complete bipartite graph $K_n \big[ X,V(K_t) \big]$ yields a red copy of $H$ containing an NIM-$H$ edge, a contradiction.
Analogously, the blue NIM-$H$ edges contained in $X$ induce a graph which is $H_1$-free.
Hence,
\[\begin{split}
\big| E(c,H) \big|\le &\ex \big( |X|,H_1 \big) +\ex \big( n-|X|,H_1 \big)\\
\le &\ex(n,H_1)\le \ex(n,H),
\end{split}\]
where the second inequality holds since $H_1$ is connected. The proof is complete in this case.

{\bf Case 2:} $|X|=a$.  By Claim \ref{claim2},  the set of red NIM-$H$ edges can be partitioned into two parts: the ones contained in $V(K_t)\cup Y$ and the remaining ones which are contained in $X$. 
Since all blue NIM-$H$ edges are contained in~$X$ by Claim~\ref{claim1}, the sum of the total number of blue NIM-$H$ edges and the number of red NIM-$H$ edges contained in $X$ is at most $\binom{a}{2}$. 
The set of red NIM-$H$ edges contained in $V(K_t)\cup Y$ yields an $H_1$-free graph. Indeed, otherwise together with a red copy of $H_2\cup\dots \cup H_q$ (if $H$ is disconnected) in  $K_n \big[ X,V(K_t) \big]$, we could find a red copy of $H$ containing a red NIM-$H$ edge, a contradiction. 
Thus the number of red NIM-$H$ edges contained in $V(K_t)\cup Y$ is at most $\ex(n-a,H_1)$.

Therefore, the total number of NIM-$H$ edges is at most $\ex(n-a,H_1)+\binom{a}{2}$. 
Since $H_1$ is connected and contains a tail, it follows that the union of a star $S_{a-1}$ on $a$ vertices    and an extremal graph for $\ex(n-a,H_1)$ is still $H_1$-free. Hence,
\[\ex(n-a,H_1)+(a-1)\le \ex(n,H_1).\]
Thus, we have
\[\begin{split}
\big| E(c,H) \big| \le &\ex(n-a,H_1)+\binom{a}{2}\le \ex(n,H_1)+\binom{a-1}{2}\\
\le &\ex(n,H)+\binom{a-1}{2},
\end{split}\]
and the proof of this case is complete.

{\bf Case 3:} $|X|\le a-1$. By Claim \ref{claim1}, the number of blue NIM-$H$ edges is at most $\binom{a-1}{2}$, and the red NIM-$H$ edges yield an $H$-free graph. Hence
\[ \big| E(c,H) \big| \le \ex(n,H)+\binom{a-1}{2},\]
and the proof is complete. $\hfill\blacksquare$

\begin{remark}
In \cite{zhu2022tur}, the first author and Chen also give a family of examples such that $\chi(H)=3$ and $f(n,H)>\ex(n,H)$.  
\end{remark}

\section{Proof of Theorem \ref{Thm3}} \label{sec4}
We first give a $2k$-edge-coloring of $K_n$ with $(2k-1)\ex(n, P_{2k})+(k-1)\binom{2k-1}{2}$ NIM-$P_{2k}$ edges when $2k-1$ is a prime and $n\in \big\{ a(2k-1)+(k-1), \, a(2k-1)+k \big\}$. Before showing our construction, we need to recall the exact value of $\ex(n,P_\ell)$. 
\begin{theorem}[Faudree and Schelp~\cite{Faudree}]\label{Faudree}
Let $n=a(\ell-1)+b$ with $0\le b\le \ell-2$. Then we have
$$\ex(n,P_{\ell})=a\binom{\ell-1}{2}+\binom{b}{2}.$$
If $\ell$ is even and $b\in \{\ell/2,\ell/2-1\}$, then the extremal graphs are $tK_{\ell-1}\cup \big( K_{\ell/2-1}+\overline{K}_{n-t(\ell-1)-\ell/2+1} \big)$ for any $0 \le t\le a$. Otherwise $aK_{\ell-1}\cup K_b$ is the unique extremal graph.
\end{theorem}
Therefore, by Theorem~\ref{Faudree}, when $n\in \big\{ a(2k-1)+(k-1), \, a(2k-1)+k \big\}$, the extremal graphs for $\ex(n,P_{2k})$ are $tK_{2k-1}\cup \big( K_{k-1}+\overline{K}_{n-t(2k-1)-(k-1)} \big)$ for any $0\le t\le a$.

\vskip 3mm
 Let $U$ be a subset of size $(2k-1)^2$ of $V(K_n)$ and label the vertices  of $U$ by $[i,j]$ where $1\le i,j\le 2k-1$. 
 We divide $U$ into $2k-1$ subsets by setting $$U_i= \big\{ [i,1],~[i,2],\ldots,~[i,2k-1] \big\},~~ 1\le i\le 2k-1.$$
When it is not confusing, we also let $U$ and $U_i$ denote the cliques induced by the vertices in them. 

For any $1 \le i,j \le 2k-1$, let $\sigma_{ji}$ denote the clique induced by the vertices $[1,i], [2,{i+j}], \ldots, \allowbreak [2k-1, i+(2k-2)j]$, where the indices are taken modulo $2k-1$. For any $1 \le j \le 2k-1$, let
\[ \mathcal{C}_j = \{ \sigma_{ji} : 1 \le i \le 2k-1 \}. \]
Then $\mathcal{C}_j$ is a set consisting of $2k-1$ disjoint $(2k-1)$-cliques.


Let $c: E(K_n)\rightarrow \{c_1,\ldots,c_{2k}\}$ be a $2k$-edge-coloring  defined as follows. Let $W=V(K_n)\setminus U$.  For any $j\in [2k-1]$, we assign the color $c_j$ to the edges of each clique $\sigma_{ji}$ in $\mathcal{C}_j$. Let $\sigma_{j1}^\diamond$ denote the clique induced by the vertices $[k+1,1+kj], \ldots, [2k-1,1+(2k-2)j]$. Clearly, we have $\sigma_{j1}^\diamond \subset \sigma_{j1}$. Now consider the sub-clique $\sigma_{j1}-\sigma_{j1}^\diamond$ and replace the color $c_j$ by $c_{2k}$ inside it. With this, $\sigma_{j1}$ decomposes into a copy of $K_{k-1}+\overline{K}_{k}$ colored by $c_j$ and a copy of $K_{k}$ colored by $c_{2k}$. After this, we assign the color $c_j$ to all the edges between $\sigma_{j1}^\diamond$ and $V$. Figure~\ref{fig:coloring} shows the subgraph induced by the edges colored by~$c_{2k-1}$. Finally, we assign the color $c_{2k}$ to the edges which have not been colored yet.


\begin{figure}[h]
 \begin{center}
 \begin{tikzpicture}[scale=1.5]
  \tikzstyle{vertex}=[draw,circle,fill=black,minimum size=4,inner sep=0]
  
  \draw[very thick, fill=black!20] (-0.5,0.9) ellipse (0.15 and 0.5);
  
  \node[vertex] (v11) at (-0.5,2.25) {};
  \draw[fill] (-0.5,2) circle (0.3pt);
  \draw[fill] (-0.5,1.9) circle (0.3pt);
  \draw[fill] (-0.5,1.8) circle (0.3pt);
  \node[vertex] (v12) at (-0.5,1.55) {};
  \node[vertex] (v13) at (-0.5,1.25) {};
  \draw[fill] (-0.5,1) circle (0.3pt);
  \draw[fill] (-0.5,0.9) circle (0.3pt);
  \draw[fill] (-0.5,0.8) circle (0.3pt);
  \node[vertex] (v14) at (-0.5,0.55) {};
  
  \draw[very thick] (v11) to [bend left=70] (v13);
  \draw[very thick] (v11) to [bend left=80] (v14);
  \draw[very thick] (v12) to [bend left=50] (v13);
  \draw[very thick] (v12) to [bend left=70] (v14);
  
  \node at (-0.5,0.2) {$\sigma_{2k-1,1}^{\diamond}$};
  
  \draw[very thick, fill=black!20] (0.5,1.4) ellipse (0.17 and 1);
  \node at (0.5,2.8) {$\sigma_{2k-1,2}$};
  
  \node[vertex] (v21) at (0.5,2.25) {};
  \draw[fill] (0.5,2) circle (0.3pt);
  \draw[fill] (0.5,1.9) circle (0.3pt);
  \draw[fill] (0.5,1.8) circle (0.3pt);
  \node[vertex] (v22) at (0.5,1.55) {};
  \node[vertex] (v23) at (0.5,1.25) {};
  \draw[fill] (0.5,1) circle (0.3pt);
  \draw[fill] (0.5,0.9) circle (0.3pt);
  \draw[fill] (0.5,0.8) circle (0.3pt);
  \node[vertex] (v24) at (0.5,0.55) {};
  
  \draw[fill] (0.9,2.25) circle (0.3pt);
  \draw[fill] (1,2.25) circle (0.3pt);
  \draw[fill] (1.1,2.25) circle (0.3pt);
  
  \draw[fill] (0.9,1.55) circle (0.3pt);
  \draw[fill] (1,1.55) circle (0.3pt);
  \draw[fill] (1.1,1.55) circle (0.3pt);
  
  \draw[fill] (0.9,1.25) circle (0.3pt);
  \draw[fill] (1,1.25) circle (0.3pt);
  \draw[fill] (1.1,1.25) circle (0.3pt);
  
  \draw[fill] (0.9,0.55) circle (0.3pt);
  \draw[fill] (1,0.55) circle (0.3pt);
  \draw[fill] (1.1,0.55) circle (0.3pt);

  \draw[very thick, fill=black!20] (1.5,1.4) ellipse (0.17 and 1);
  \node at (1.5,2.65) {$\sigma_{2k-1,k+1}$};
  
  \node[vertex] (v31) at (1.5,2.25) {};
  \draw[fill] (1.5,2) circle (0.3pt);
  \draw[fill] (1.5,1.9) circle (0.3pt);
  \draw[fill] (1.5,1.8) circle (0.3pt);
  \node[vertex] (v32) at (1.5,1.55) {};
  \node[vertex] (v33) at (1.5,1.25) {};
  \draw[fill] (1.5,1) circle (0.3pt);
  \draw[fill] (1.5,0.9) circle (0.3pt);
  \draw[fill] (1.5,0.8) circle (0.3pt);
  \node[vertex] (v34) at (1.5,0.55) {};
  
  \draw[fill] (1.9,2.25) circle (0.3pt);
  \draw[fill] (2,2.25) circle (0.3pt);
  \draw[fill] (2.1,2.25) circle (0.3pt);
  
  \draw[fill] (1.9,1.55) circle (0.3pt);
  \draw[fill] (2,1.55) circle (0.3pt);
  \draw[fill] (2.1,1.55) circle (0.3pt);
  
  \draw[fill] (1.9,1.25) circle (0.3pt);
  \draw[fill] (2,1.25) circle (0.3pt);
  \draw[fill] (2.1,1.25) circle (0.3pt);
  
  \draw[fill] (1.9,0.55) circle (0.3pt);
  \draw[fill] (2,0.55) circle (0.3pt);
  \draw[fill] (2.1,0.55) circle (0.3pt);

  \draw[very thick, fill=black!20] (2.5,1.4) ellipse (0.17 and 1);
  \node at (2.5,2.8) {$\sigma_{2k-1,2k-1}$};
  
  \node[vertex] (v41) at (2.5,2.25) [label={[xshift=29pt, yshift=-10pt] \small{$[1, 2k-1]$}}] {};
  \draw[fill] (2.5,2) circle (0.3pt);
  \draw[fill] (2.5,1.9) circle (0.3pt);
  \draw[fill] (2.5,1.8) circle (0.3pt);
  \node[vertex] (v42) at (2.5,1.55) [label={[xshift=31pt, yshift=-10pt] \small{$[k, 2k-1]$}}] {};
  \node[vertex] (v43) at (2.5,1.25) [label={[xshift=39pt, yshift=-10pt] \small{$[k+1, 2k-1]$}}] {};
  \draw[fill] (2.5,1) circle (0.3pt);
  \draw[fill] (2.5,0.9) circle (0.3pt);
  \draw[fill] (2.5,0.8) circle (0.3pt);
  \node[vertex] (v44) at (2.5,0.55) [label={[xshift=41pt, yshift=-10pt] \small{$[2k+1,2k-1]$}}] {};

  \node at (-2,3) {$W$};
  
  \node[vertex] (w1) at (-2,2.5) {};
  \node[vertex] (w2) at (-2,2) {};
  \draw[fill] (-2,1.5) circle (0.3pt);
  \draw[fill] (-2,1.4) circle (0.3pt);
  \draw[fill] (-2,1.3) circle (0.3pt);
  \node[vertex] (w3) at (-2,0.8) {};
  \node[vertex] (w4) at (-2,0.3) {};
  
  \draw[very thick] (w1) -- (v13);
  \draw[very thick] (w1) -- (v14);
  \draw[very thick] (w2) -- (v13);
  \draw[very thick] (w2) -- (v14);
  \draw[very thick] (w3) -- (v13);
  \draw[very thick] (w3) -- (v14);
  \draw[very thick] (w4) -- (v13);
  \draw[very thick] (w4) -- (v14);
 \end{tikzpicture}
 \caption{The subgraph induced by the edges of color $c_{2k-1}$.} \label{fig:coloring}
 \end{center}
\end{figure}
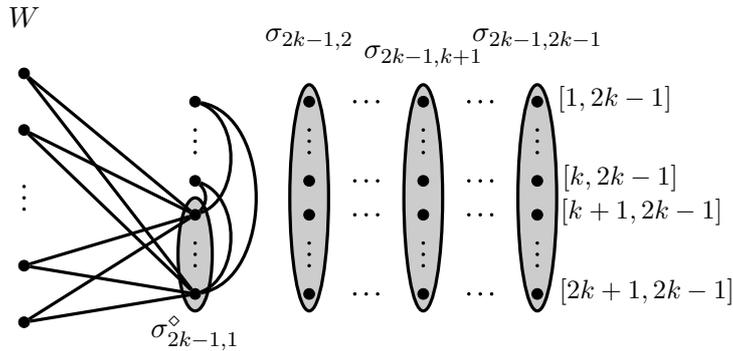

In the next two paragraphs, we show that this $2k$-edge-coloring is well-defined, namely, each edge is assigned exactly one color. Clearly, each edges is assigned at least one color and the edges inside $W$ or between $U_1 \cup \dots \cup U_k$ and $W$ are assigned exactly one color.

Note that $U$ is a $(2k-1)^2$-clique. Let $1 \le i,\ell,s,t \le 2k-1$. Clearly, the edge $[i,s][i,t]$ is only covered by the clique $U_i$. If the edge $[i,s][\ell,t]$ with $i<\ell$ were covered by two cliques, say by one in $\mathcal{C}_j$ and by another one in $\mathcal{C}_{j'}$ for some $1\le j,j'\le 2k-1$, then
\[ \begin{cases}
    t\equiv s+(\ell-i)j & \pmod{2k-1} \\
    t\equiv s+(\ell-i)j' & \pmod{2k-1}
   \end{cases}
\]
would hold, and
since $2k-1$ is prime, we would have $j=j'$, a contradiction. Thus, each edge inside $U$ is covered by at most one clique in $\mathcal{C}_j$ or by the clique $U_i$. On the other hand, considering the number of edges in $U$ and the total number of edges of cliques in each $\mathcal{C}_j$ and $U_i$ yields
$$e(U)=\sum_{i=1}^{2k-1}e(U_i)+\sum_{j=1}^{2k-1}\sum_{\sigma_{ji}\in\mathcal{C}_j}e(\sigma_{ji}).$$
Therefore, the cliques in each $\mathcal{C}_j$ together with the cliques $U_i$ for all $1 \le i,j \le 2k-1$ form an edge-decomposition of the large clique $U$. Hence each edge in $U$ is assigned one color.

Now we show that for any $1\le j,j'\le 2k-1$ with $j \ne j'$, the sub-cliques $\sigma_{j1}^\diamond$ and $\sigma_{j'1}^\diamond$ are vertex-disjoint. Supposing that a vertex $[i,1+(i-1)j]\in V(\sigma_{j1}^\diamond)$ is also contained in $\sigma_{j'1}^\diamond$ for some $1\le i,j,j'\le 2k-1$, we obtain
$$1+(i-1)j\equiv 1+(i-1)j' \pmod{2k-1}.$$
Since $2k-1$ is a prime number, we get $j=j'$, a contradiction. Thus the sub-cliques  $\sigma_{j1}^\diamond$ for all $1 \le j\le 2k-1$ form a vertex-decomposition of $U_{k+1}\cup \dots \cup U_{2k-1}$. Hence, each edge between $U_{k+1}\cup \dots \cup U_{2k-1}$ and $W$ is assigned one color in $\{c_1,\ldots,c_{2k-1}\}$. Therefore, our $2k$-edge-coloring $c$ is well-defined.

Note that for any $1\le j\le 2k-1$, the subgraph induced by the edges of color $c_j$ is a copy of $tK_{2k-1}\cup \big(K_{k-1}+\overline{K}_{n-(k-1)-t(2k-1)} \big)$ with $t=2k-2$, and this graph is extremal for $\ex(n,P_{2k})$ when $n\in \big\{a(2k-1)+(k-1), \, a(2k-1)+k \big\}$. Now consider the edges colored by $c_{2k}$. They are in the cliques $U_i$ with $1 \le i \le 2k-1$, inside $\sigma_{j1}-\sigma_{j1}^\diamond$ with $1 \le j\le 2k-1$, inside $W$, and between $U_1\cup \dots \cup U_k$ and $W$. Note that for any $k+1\le i\le 2k-1$, $U_i$ are independent $(2k-1)$-cliques colored by $c_{2k}$, hence the edges in $U_i$ are also NIM-$P_{2k}$. For all other $c_{2k}$-edges, they construct a large connected component such that $W$ is a clique in the component. Hence none of these edges are NIM-$P_{2k}$.

Therefore,
\[ \big| E(c,P_{2k}) \big| = (2k-1)\ex(n,P_{2k})+(k-1)\binom{2k-1}{2},\]
and we are done.
\vskip 3mm

\begin{remark}
Note that $tK_{2k-1}\cup \big( K_{k-1}+\overline{K}_{n-(k-1)-t(2k-1)} \big)$ is not extremal for $\ex(n,P_{2k})$ when $n\notin \big\{ a(2k-1)+(k-1), a(2k-1)+k \big\}$, but we still have
$$\ex(n,P_{2k})-e\Big(tK_{2k-1}\cup \big(K_{k-1}+\overline{K}_{n-(k-1)-t(2k-1)} \big)\Big)<(k-1)^2.$$
Hence in our construction, when $2k-1$ is prime, the number of NIM-$P_{2k}$ edges is more than $(2k-1)\ex(n,P_{2k})$. That is to say, when $2k-1$ is prime, we have $f_{2k}(n,P_{2k})> (2k-1)\ex(n,P_{2k})$ for every sufficiently large $n$.
\end{remark}

\vskip 3mm

Next we prove the upper bound of $f_{2k}(n,P_{2k})$. Let $c: ~E(K_n)\rightarrow \{c_1,\ldots,c_{2k}\}$ be a $2k$-edge-coloring of $K_n$. We call an edge a $c_i$-edge if it is of color $c_i$ and we let $G_i$ denote the subgraph induced by all $c_i$-edges, for any $1 \le i \le 2k$. 
Without loss of generality, we can assume $e(G_{2k})\ge \binom{n}{2}/2k$.
By Theorem~\ref{Faudree}, there is a path $P$ of at least $\frac{n}{2k}$ vertices in $G_{2k}$.
Let $G_{2k}'$ be the component of $G_{2k}$ which contains the path $P$, and let $X=V(G_{2k}')$ and $Y=V(K_n)-X$. Then we have $|X|\ge \frac{n}{2k}$ and there is no $c_{2k}$-edge between $X$ and $Y$.
Since the component $G_{2k}'$ contains a long path $P$, each edge of $G_{2k}'$ is contained in a monochromatic copy of $P_{2k}$.
Hence, all NIM-$P_{2k}$ $c_{2k}$-edges are contained in $Y$.

For each  $1\le i\le 2k-1$, there are at most $\ex(n,P_{2k})$ NIM-$P_{2k}$ $c_i$-edges.   
If $|Y|\le (k-1)(2k-1)$, then there are at most
$\ex\big(|Y|,P_{2k}\big)\le(k-1)\binom{2k-1}{2}$
NIM-$P_{2k}$ $c_{2k}$-edges. 
Hence, the total number of NIM-$P_{2k}$ edges is at most
\[(2k-1)\ex(n,P_{2k})+(k-1)\binom{2k-1}{2},\]
so we are done.
Therefore, we may assume $|Y|\ge (k-1)(2k-1)+1$.
\vskip 3mm
Let us define a procedure to find pairs $(X_i,Y_i)$ satisfying the following conditions:
\begin{enumerate}[itemsep=0pt]
  \item[(i)] $X_i\subseteq X$ with $|X_i|=2k$ and $Y_i\subseteq Y$ with $|Y_i|=k$ for any $1 \le i \le 2k$,
  \item[(ii)] $Y_i$ and $Y_j$ are disjoint for any $1 \le i,j \le 2k$ with $i \ne j$,
  \item[(iii)] $K_n[X_i,Y_i]$ forms a monochromatic copy of complete bipartite graph for any $1 \le i \le 2k$.
\end{enumerate}

Assume that for some $1 \le i \le 2k$, we have found $(X_1,Y_1),\ldots, (X_{i-1},Y_{i-1})$ which satisfy the conditions. Let $s=(k-1)(2k-1)+1$. If 
\[ \left| Y \setminus \bigcup_{j=1}^{i-1} Y_j \right| \le s-1, \]
then the procedure terminates. Otherwise we choose a subset $Y_i'$ of $Y \setminus \bigcup_{j=1}^{i-1} Y_j$ with $|Y_i'|=s$. Let $Y'_i=\{y_1,\ldots,y_{s}\}$. For each $x\in X$, we define a vector $\vec{\epsilon}(x, Y_i')=(\epsilon_1,\ldots,\epsilon_{s})$ as follows: for any $1 \le j \le s$, let
$$\epsilon_j = i~~\text{if and only if the edge $xy_j$ is colored by}~ c_i.$$
Since no edge between $X$ and $Y$ is colored by $c_{2k}$, we have $\vec{\epsilon}(x,Y_i')\in \{1,\ldots,2k-1\}^{s}$ for any $x \in X$.
For each $\vec{v}\in \{1,\ldots,2k-1\}^{s}$, let $X_{\vec{v}}$ denote the set of vertices $x\in X$ for which $\vec{\epsilon}(x, Y_i')=\vec{v}$.
Hence, $X$ is divided into $(2k-1)^s$ subsets and clearly, at least one subset, say $X_{\vec{v}_i}$, contains at least $|X|/(2k-1)^s$ vertices. Observe that $K_n[X_{\vec{v}_i},y_j]$ is a monochromatic star for any $y_j\in Y_i'$. Since $|Y_i'|=(k-1)(2k-1)+1$ and there are at most $2k-1$ different colors between $X_{\vec{v}_i}$ and $Y_i'$, by pigeonhole principle, there exists a subset $Y_i\subset Y_i'$ such that $|Y_i|=k$ and the edges between $Y_i$ and $X_{\vec{v}_i}$ are monochromatic. That is $K_n[X_{\vec{v}_i},Y_i]$ is a monochromatic complete bipartite graph. Since $n\ge (2k)^{2k^2}$,
$$|X_{\vec{v}_i}|\ge \frac{|X|}{(2k-1)^s}\ge \frac{n}{(2k)^s}\ge 2k.$$
We can choose a subset $X_i$ from $X_{\vec{v}_i}$ with $|X_i|=2k$, thereby finding the pair $(X_i,Y_i)$ as we wanted.

\vskip 3mm
Note that since $Y$ is finite, the procedure terminates. Let $t$ denote the number of steps the algorithm took, and let $(X_1,Y_1),\ldots,(X_t,Y_t)$ be the pairs the algorithm found. Let $Y_0=Y\setminus \bigcup_1^{t} Y_i$. Then we have $|Y_0|\le (k-1)(2k-1)$. For any $1 \le i \le 2k-1$, let $t_i$ denote the number of the pairs $(X_j, Y_j)$ for which the edges of $K_n[X_j, Y_j]$ are of color $c_i$. Without loss of generality, we may assume that $t_1, \ldots, t_h > 0$ for some $1 \le h \le 2k-1$. Then $t=\sum_{i=1}^h t_i$. Let $1 \le i \le h$ and consider the $c_i$-edges. Without loss of generality, we can assume that $K_n[X_1,Y_1],\ldots,K_n[X_{t_i},Y_{t_i}]$ are of color $c_i$. Then each NIM-$P_{2k}$ $c_i$-edge is contained
in $V(K_n) \setminus \bigcup_{j=1}^{t_i}(X_j\cup Y_j)$. Since the sets $Y_1, \ldots, Y_{t_i}$ are pairwise disjoint and $X_1, \ldots, X_{t_i} \subseteq X$, we have
\[ \left| \bigcup_{j=1}^{t_i} (X_j\cup Y_j) \right| \ge t_i k+2k, \]
thus the number of NIM-$P_{2k}$ $c_i$-edges is at most $\ex(n-t_ik-2k, \, P_{2k})$. Now let $h+1 \le i \le 2k-1$ (if such an index exists). Since $t_i = 0$, the number of NIM-$P_{2k}$ $c_i$-edges is at most $\ex(n,P_{2k})$. 

As we have proved,  all NIM-$P_{2k}$ $c_{2k}$-edges are contained in $Y$ and $|Y|\le (k-1)(2k-1)+tk$. Therefore, the total number of NIM-$P_{2k}$ edges is at most
\begin{equation}\label{eq4.1}
\ex\big((k-1)(2k-1)+tk,P_{2k}\big)+\sum_{i=1}^h \ex(n-t_ik-2k,P_{2k})+(2k-1-h)\ex(n,P_{2k}).
\end{equation}
To prove the final result, we need the following lemma.
\begin{lemma}
Let $n_1$, $n_2$ and $c$ be  constants. Then we have
$$\ex(n_1,P_{\ell})+\ex(n_2,P_{\ell})< \ex(n_1-c,P_{\ell})+\ex(n_2+c+\ell,P_{\ell}).$$
\end{lemma}
\begin{proof} Let $n_1-c=a_1(\ell-1)+b_1$ and $n_2+c=a_2(\ell-1)+b_2$, where $0\le b_1,b_2\le\ell-2$. By Theorem~\ref{Faudree}, we have
\[\begin{split}
\ex(n_1-c,P_{\ell})&+\ex(n_2+c+\ell,P_{\ell})\ge \ex(n_1-c,P_{\ell})+\ex(n_2+c,P_{\ell})+\ex(\ell,P_{\ell}) \\
&>(a_1+a_2)\binom{\ell-1}{2}+\binom{b_1}{2}+\binom{b_2}{2}+\binom{\ell-1}{2}\\
\end{split}\]
and
\[\ex(n_1,P_{\ell})+\ex(n_2,P_{\ell})\le \frac{\ell-2}{2}(n_1+n_2)=(a_1+a_2)\binom{\ell-1}{2}+(b_1+b_2)\frac{\ell-2}{2}.\]
Hence we have
\[\begin{split}
\ex(n_1-c,P_{\ell})&+\ex(n_2+c+\ell,P_{\ell})- \big( \ex(n_1,P_{\ell})+\ex(n_2,P_{\ell}) \big)\\
&>\binom{b_1}{2}+\binom{b_2}{2}+\binom{\ell-1}{2}-(b_1+b_2)\frac{\ell-2}{2} >0.
\end{split}\]
we are done. 
\end{proof}

\vskip 3mm
When applying the above lemma to~\eqref{eq4.1}, we get
\begin{align*}
\ex\big((k-1)(2k-1)+tk,P_{2k}\big)&+\sum_{i=1}^s \ex(n-t_ik-2k,P_{2k}) +(2k-1-s)\ex(n,P_{2k})\\
<&(2k-1)\ex(n,P_{2k})+\ex\big((k-1)(2k-1),P_{2k}\big)\\
=&(2k-1)\ex(n,P_{2k})+(k-1)\binom{2k-1}{2}.
\end{align*}
Thus the proof is complete. $\hfill\blacksquare$

\section{Acknowledgements}
The research of the authors Gy\H{o}ri and Salia was partially supported by the National Research, Development and Innovation Office NKFIH, grants  K132696, SNN-135643 and K126853. 
The research of Salia was supported by the Institute for Basic Science (IBS-R029-C4). 
The research of Tompkins was supported by NKFIH grant K135800.

\bibliography{ref.bib}

\begin{thebibliography}{10}

\bibitem{Kettle}
N.~Bushaw and N.~Kettle.
\newblock Tur\'{a}n numbers of multiple paths and equibipartite forests.
\newblock {\em Combin. Probab. Comput.}, 20:837--553, 2011.

\bibitem{Erdos}
P.~Erd\H{o}s.
\newblock Some recent problems and results in graph theory.
\newblock {\em Discrete Math.}, 164:81--85, 1997.

\bibitem{Faudree}
R.J. Faudree and R.H. Schelp.
\newblock Path {Ramsey} numbers in multicolourings.
\newblock {\em J. Combin. Theory B.}, 19:150--160, 1975.

\bibitem{Goodman}
A.W. Goodman.
\newblock On sets of acquaintances and strangers at any party.
\newblock {\em Amer. Math. Monthly}, 66:778--783, 1959.

\bibitem{Keevash}
P.~Keevash and B.~Sudakov.
\newblock On the number of edges not covered by monochromatic copies of a fixed
  graph.
\newblock {\em J. Combin. Theory Ser. B}, 90:41--53, 2004.

\bibitem{Liu}
H.~Liu, O.~Pikhurko, and M.~Sharifzadeh.
\newblock Edges not in any monochromatic copy of a fixed graph.
\newblock {\em J. Combin. Theory Ser. B}, 135:16--43, 2019.

\bibitem{Ma}
J.~Ma.
\newblock On edges not in monochromatic copies of a fixed bipartite graph.
\newblock {\em J. Combin. Theory Ser. B}, 123:240--248, 2017.

\bibitem{mclennan2005erdHos}
A.~McLennan.
\newblock The {Erd{\H{o}}s-S{\'o}s Conjecture} for trees of diameter four.
\newblock {\em Journal of Graph Theory}, 49(4):291--301, 2005.

\bibitem{Pyber}
L.~Pyber.
\newblock Clique covering of graphs.
\newblock {\em Combinatorica}, 6:393--398, 1986.

\bibitem{Thomason}
A.~Thomason.
\newblock Graph products and monochromatic multiplicities.
\newblock {\em Combinatorica}, 17:125--134, 1997.

\bibitem{Yuan}
L.T. Yuan.
\newblock Extremal graphs for edge blow-up of graphs.
\newblock {\em J. Combin. Theory Ser. B}, 152:379--398, 2022.

\bibitem{zhu2022tur}
Xiutao Zhu and Yaojun Chen.
\newblock Tur{\'a}n number for odd-ballooning of trees.
\newblock {\em arXiv preprint arXiv:2207.11506}, 2022.

\end{thebibliography}

\end{document}